\documentclass[a4paper]{amsart}
\usepackage{amssymb}
\usepackage{mathrsfs}
\usepackage{color}
\usepackage{epsfig}
\usepackage{graphicx}
\usepackage[all]{xy}

\theoremstyle{plain}
\newtheorem{thm}{Theorem}[section]
\newtheorem{prp}[thm]{Proposition}
\newtheorem{crl}[thm]{Corollary}
\newtheorem{lmm}[thm]{Lemma}

\newtheorem{question}[thm]{Question}

\theoremstyle{definition}
\newtheorem{rem}[thm]{Remark}

\newtheorem{dfn}[thm]{Definition}
\newtheorem{exa}[thm]{Example}

\newcommand{\C}{{\mathbb C}}
\newcommand{\R}{{\mathbb R}}

\newcommand{\Z}{{\mathbb Z}}

\def\t{\mathfrak t}

\def\CO{\mathcal{O}}

\def\<{\left\langle}
\def\>{\right\rangle}

\def\End{\operatorname{End}}

\def\ind{\operatorname{ind}}

\newcommand{\abs}[1]{\lvert#1\rvert}
\newcommand{\norm}[1]{\lVert#1\rVert}

\numberwithin{equation}{section}

\title{Equivariant local index}

 \author[T. Yoshida]{Takahiko Yoshida}

\thanks{Partly supported by Grant-in-Aid for Young Scientists (B) 22740046.}

\address{Department of Mathematics, School of Engineering, Tokyo Denki University, 2-2 Kanda-Nishiki-cho, Chiyoda-ku, Tokyo, 101-8457, Japan}
\email{takahiko@mail.dendai.ac.jp}

\begin{document}
\maketitle

\begin{abstract}
This is an expository article on the equivariant local index developed by Fujita, Furuta, and the author in ~\cite{FFY3}. 
\end{abstract}

\section{Background, motivation, and purpose}
This is based on a joint work with Hajime Fujita and Mikio Furuta~\cite{FFY, FFY2, FFY3,FFY4}. In~\cite{FFY, FFY2} we developed an index theory for Dirac-type operators on possibly noncompact Riemannian manifolds and applied the index theory to the geometric quantization of Lagrangian fibrations. %in particular, the relationship between the Riemann-Roch index and Bohr-Sommerfeld fibers.  
 In \cite{FFY3} we refined the index theory in the case of torus actions. As an application we obtained a proof of the quantization conjecture, concering the commutativity of the quantization and the symplectic reduction, in the case of torus actions. Other applications of the index theory will be described in \cite{FFY4,FY5}. The purpose of this note is to explain the equivariant index theory developed in \cite{FFY3} in a simple symplectic case. 

Let us recall the background and our motivation. One of our motivation comes from the geometric quantization. Let $(M,\omega)$ be a closed symplectic manifold. Suppose that $(M,\omega)$ is prequantizable, namely, the cohomology class represented by $\omega$ is in the image of the natural map $H^2(M;\Z)\to H^2(M;\R)$. Then, there exists a Hermitian line bundle $L\to M$ with Hermitian connection $\nabla^L$ whose curvature form $F_{\nabla^L}$ is equal to $-2\pi\sqrt{-1}\omega$.  $(L,\nabla^L)$ is called a {\it prequantum line bundle}. 

It is well known that a symplectic manifold is equipped with an almost complex structure $J$ compatible in the sense that $g(u,v):=\omega(u,Jv)$ is a Riemannian metric. For example see~\cite{McDuffSalamon}. We take and fix a compatible almost complex structure $J$. We extend $J$ to $TM\otimes_{\R} \C$ complex linearly, and $\sqrt{-1}$ and $-\sqrt{-1}$-eigenspaces by $T^{1,0}M$ and $T^{0,1}M$, respectively. %Then, the complexified tangent bundle $TM\otimes_{\R} \C$ decomposes into the eigenspaces $TM\otimes_{\R} \C=T^{1,0}M\oplus T^{0,1}M$. 
We put
\[
W:=\wedge^{0,\bullet}T^*M\otimes L=\wedge^\bullet (T^{0,1}M)^*\otimes L. 
\]
The Levi-Civita connection with respect to the Riemannian metric $g$ together with the Hermitian connection $\nabla^L$ of $L$ induces the canonical connection $\nabla\colon \Gamma (W)\to \Gamma(T^*M\otimes W)$ on $W$. Moreover, the Clifford module structure $c\colon Cl(T^*M)\to \End (W)$ is defined as
\[
c(u):=\sqrt{-2}\left(u^{0,1}\wedge \alpha -u^{0,1}\llcorner \alpha\right) 
\]
for $u\in T^*M$ and $\alpha \in W$, where $u^{0,1}$ is the $(0,1)$-factor of $u\otimes 1\in T^*M\otimes \C\cong (T^{1,0}M)^*\oplus (T^{0,1}M)^*$. Then, the {\it Spin${}^c$ Dirac operator} is defined to be the composition 
\[
D:=c\circ \nabla \colon \Gamma(W)\to \Gamma(W).
\]
It is well known that $D$ is a first order, formally self-adjoint, elliptic differential operator of degree-one, and if $(M,\omega ,J)$ is K\"ahler and $L$ is holomorphic, then $D$ is nothing but the Dolbeault operator with coefficients in $L$ up to constant, namely, $D=\sqrt{2}(\bar \partial \otimes L+\bar \partial^*\otimes L)$. 

Let $D^0$ and $D^1$ be the degree-zero and degree-one parts of $D$, namely,
\[
D^0:=D|_{\wedge^{0,even}T^*M\otimes L},\ \ D^1:=D|_{\wedge^{0,odd}T^*M\otimes L}, 
\]
respectively. Since $M$ is closed and $D$ is elliptic, $D$ is Fredholm, namely, both of the kernels of $D^0$ and $D^1$ are finite dimensional vector spaces. Then, the index of $D$ is defined by
\[
\ind D:=\dim \ker D^0-\dim \ker D^1. 
\]
$\ind D$ is called the {\it Riemann-Roch index}. Note that $\ind D$ depends only on $\omega$ and does not depend on the choice of $J$ and $\nabla^L$ since the index is homotopy invariant and the space of compatible almost complex structures of $(M,\omega)$ is contractible. By the Atiyah-Singer index theorem, $\ind D$ can be expressed as 
\[
\ind D=\int_M e^\omega Td(TM,J), 
\]
where $Td(TM,J)$ is the Todd class of the complex vector bundle $TM$ with complex structure $J$. Moreover, if $(M,\omega ,J)$ is K\"ahler and $L$ is holomorphic, then $\ind D$ is equal to the Euler-Poincar\'e characteristic
\[
\ind D=\sum_{q\ge 0}(-1)^q\dim H^q(M,\CO_L). 
\]
For the Spin${}^c$ Dirac operators see \cite{LM}. 

A {\it Lagrangian fibration} is a fiber bundle $\pi\colon (M,\omega)\to B$ from $(M,\omega)$ to a manifold $B$ whose fiber is a Lagrangian submanifold of $(M,\omega)$. Note that for a Lagrangian fibration $\pi\colon (M,\omega)\to B$, the restriction $(L,\nabla^L)|_{\pi^{-1}(b)}$ to each fiber $\pi^{-1}(b)$ is a flat line bundle since $F_{\nabla^L}=-2\pi\sqrt{-1}\omega$ and a fiber is Lagrangian. A fiber $\pi^{-1}(b)$ of a Lagrangian fibration $\pi\colon (M,\omega)\to B$ is said to be {\it Bohr-Sommerfeld} if $(L,\nabla^L)|_{\pi^{-1}(b)}$ has a non-trivial global parallel section. The Bohr-Sommerfeld condition is equivalent to that the degree zero cohomology $H^0\left(\pi^{-1}(b);(L,\nabla^L)|_{\pi^{-1}(b)}\right)$ with coefficients in the local system $(L,\nabla^L)|_{\pi^{-1}(b)}$ is non-trivial. It is known that the Bohr-Sommerfeld fibers appear discretely. Then, in \cite{Andersen} Andersen showed that for a Lagrangian fibration $\pi\colon (M,\omega)\to B$ the Riemann-Roch index is equal to the number of Bohr-Sommerfeld fibers. 
%Let $(M,\omega)$ be a prequantizable closed symplectic manifold with prequantum line bundle $(L,\nabla^L)$. Suppose $\pi\colon (M,\omega)\to B$ is a Lagrangian fibration. Then, the index of the Spin${}^c$ Dirac operator is equal to the number of Bohr-Sommerfeld fibers. 

A completely integrable system can be thought of as a Lagrangian fibration with singular fibers. Similar results are known for several completely integrable systems, such as the polygon space~\cite{Kamiyama}, the Gelfand-Cetlin completely integrable system on a complex flag variety~\cite{GuS4} and the Goldman completely integrable system on the moduli space of flat $SU(2)$-bundles on a Riemann surface~\cite{JW1}. 

Suppose $(M,\omega)$ is equipped with an effective Hamiltonian action of a compact Lie group $G$ which lifts to $L$ and preserves all the data. Then, 
%If a compact Lie group $G$ acts effectively on $(M,\omega)$ and the $G$-action lifts to $L$ that preserves all the data, then, 
$\ker D^0$ and $\ker D^1$ become $G$-representations. In this case the equivariant Riemann-Roch index is defined as 
\[
\ind_G D:=\ker D^0-\ker D^1\in R(G), 
\]
where $R(G)$ is the representation ring of $G$. In the case where $G$ is a torus $(S^1)^n$, $M$ is a complex $n$-dimensional nonsingular projective toric variety, and $L$ is an ample line bundle, it is known by Danilov~\cite{Dan} that $\ind_G D$ has the following irreducible decomposition
\begin{equation*}
\ind_G D=\bigoplus_{\gamma^* \in \mu (M)\cap \t_{\Z}^*}\C_{\gamma^*} ,
\end{equation*}
where $\mu$ is the moment map associated to $M$ and $\C_{\gamma^*}$ is the irreducible representation with weight $\gamma^*$. The moment map $\mu$ can be thought of as a Lagrangian fibration with singular fibers. All singular fibers of $\mu$ are smooth tori. Hence, in this case, the notion of a Bohr-Sommerfeld fiber makes sense even for singular fibers. Moreover, elements of $\mu(M)\cap \t_{\Z}^*$ correspond one-to-one to Bohr-Sommerfeld fibers. In particular, the Danilov formula can be thought of as a refinement of Andersen's result. 
The Danilov formula was generalized to non-symplectic cases, such as, presymplectic toric manifolds~\cite{KT}, Spin${}^c$ manifolds~\cite{GK}, and torus manifolds~\cite{Masuda}. 

In the Geometric quantization the Riemann-Roch index and the number of Bohr-Sommerfeld fibers correspond to the dimensions of the quantum Hilbert spaces obtained by the Spin${}^c$ quantization and the geometric quantization using a real polarization, respectively. From the viewpoint of the geometric quantization it is fundamental to investigate the relationship between these two quantizations. 

The above results are localization phenomena of the Riemann-Roch index to the Bohr-Sommerfeld fibers. So we have a natural question: We wonder whether all of these localization phenomena might be caused by the same mechanism. If it is true, make clear the mechanism of the phenomena. %これら一連の局所化現象は全て同じメカニズムによって引き起こされているのではないだろうか？
%We think it might be the case that all of these results are obtained by the same mechanism. If it is true, make clear the mechanism of the phenomena. 
% such that the above results are obtained as consequences of it? If it is true, make clear the mechanism of the phenomena. 
%On the basis of these phenomena it is natural to ask whether there exists a localization principle of the index of the Spin${}^c$ Dirac operator to Bohr-Sommerfeld fibers such that the above results are obtained as consequences of it.   
%
%It might be the case that there exists a localization principle of the index to Bohr-Sommerfeld fibers and the above results are obtained as consequences of it. 
%
%There might exist a localization principle of the index to Bohr-Sommerfeld fibers and the above results are obtained as consequences of it. 
%
%Do these phenomena imply that there exists a localization principle of the index to Bohr-Sommerfeld fibers and the above results are obtained as consequences of it?
%
%These phenomena suggest that there is a localization principle of the index of the Spin${}^c$ Dirac operator to Bohr-Sommerfeld fibers and the above results are obtained as consequences of it.   

For this question we gave a partial answer in \cite{FFY,FFY2,FFY3}. Namely, in \cite{FFY} we developed an index theory for Dirac-type operators on possibly noncompact Riemannian manifolds, which we call the {\it local index}, and improved the result obtained in \cite{FFY} and obtained a product formula for the local index in \cite{FFY2}.  In \cite{FFY3} we refined the local index to the case of torus actions and gave a proof of the quantization conjecture for the Hamiltonian torus actions. Some of the above results, such as the equality between the Riemann-Roch index and the number of Bohr-Sommerfeld fibers for a nonsingular Lagrangian fibration and the Danilov formula for a toric variety, were obtained as consequences of the excision property for the (equivariant) local index. 
See \cite{FFY,FFY2,FFY3,FY5}. 

Although the index theory developed in \cite{FFY,FFY2,FFY3} is formulated for Riemannian manifolds it seems complicated for non experts. So, in this note, for simplicity, we will explain the equivariant version of the local index for the Hamiltonian $S^1$-actions. It is one of the simplest case. 

This note is organized as follows. In Section 2 we will explain two versions of local indices for prequantizable symplectic manifolds with Hamiltonian $S^1$-actions. In particular, when the manifold has an open covering that satisfies certain conditions we can obtain localization formulas for local indices. In Section 3, we will give a way to take such an open covering and will consider the localization formulas for the open covering in details.

%\section{Equivariant Riemann-Roch index}

\section{Equivariant local index}
\subsection{$\ind_{S^1}(M,V;L)$}
In this subsection let us recall the equivariant local index in the symplectic case. Let $(M,\omega)$ be a possibly non-compact symplectic manifold and $(L,\nabla^L)\to (M,\omega)$ a prequantum line bundle on it. Suppose $M$ is equipped with an effective Hamiltonian $S^1$-action which lifts to $L$ and preserves all the data. Note that each orbit $\CO$ is isotropic in the sense that $\omega|_{\CO}\equiv 0$. In particular, the restriction of $(L,\nabla^L)$ to each orbit is a flat line bundle because of $\frac{\sqrt{-1}}{2\pi}F_{\nabla^L}=\omega$. 

In order to define the equivariant local index we introduce the following notion. 
\begin{dfn}%[$L$-acyclic orbit]
An orbit $\CO$ is said to be {\it $L$-acyclic} if it satisfies the condition $H^0\left(\CO ;(L,\nabla^L)|_{\CO}\right)=0$.
\end{dfn}
Note that the non $L$-acyclic condition is a generalization of the Bohr-Sommerfeld condition for Lagrangian submanifolds. 

The following lemma is one of the key points to define the equivariant local index. 
\begin{lmm}\label{key}
Let $\CO$ be an orbit of the $S^1$-action on $M$. Then, the following conditions are equivalent:
\begin{enumerate}
\item $\CO$ is $L$-acyclic. 
\item $(L,\nabla^L)|_{\CO}$ admits no non-trivial global parallel section. 
\item $H^\bullet \left(\CO;(L,\nabla^L)|_{\CO}\right)=0$. 
\item The kernel of the de Rham operator of $\CO$ with coefficients in $L$ vanishes. 
\end{enumerate}
\end{lmm}
\begin{rem}
An orbit consisting of a fixed point is not $L$-acyclic since on such an orbit $(L,\nabla^L)$ always has a non-trivial global parallel section. 
\end{rem}
\begin{proof}[Proof of Lemma~\ref{key}]
It is clear that the first two conditions are equivalent. Since an $L$-acyclic orbit $\CO$ is a circle the first and third conditions are equivalent. See~\cite[Lemma~2.29]{FFY2}. Moreover, by the Hodge theory, the third condition is also equivalent to the fourth condition.  
\end{proof}
\begin{exa}[Non $L$-acyclic orbits in $\C P^1$]\label{non L-acyclic orbits in CP^1}
Let $k$ be a positive integer. %The $S^1$ acts on $\left(\C^2, \frac{\sqrt{-1}}{2\pi}\sum_{j=0}^1dz_j\wedge d\bar z_j\right)$ and on the prequantum line bundle $\left(\C^2\times \C, d+\frac{1}{2}\sum_{j=0}^1(z_jd\bar z_j-\bar z_jdz_j)\right)\to \left(\C^2, \frac{\sqrt{-1}}{2\pi}\sum_{j=0}^1dz_j\wedge d\bar z_j\right)$ by  
%\[
%h(z_0,z_1):=(hz_0,hz_1),\ \ h(z_0,z_1,v):=(hz_0,hz_1,h^kv)
%\]
%for $h\in S^1$, $(z_0,z_1)\in \C^2$, and $(z_0,z_1,v)\in \C^2\times\C$, respectively. 
%
%
Define $(M,\omega)$ and $(L,\nabla^L)$ to be the following quotient spaces by the equivalence relations
\[
\begin{split}
&(M,\omega):=
\left(S^3_k,\frac{\sqrt{-1}}{2\pi}\sum_{j=0}^1dz_j\wedge d\bar z_j\right)/_{(z_0,z_1)\sim (hz_0,hz_1)\ \ (h\in S^1)}, \\
&(L,\nabla^L):=
\left(S^3_k\times \C , d+\frac{1}{2}\sum_{j=0}^1(z_jd\bar z_j-\bar z_jdz_j)\right)/_{(z_0,z_1,v)\sim (hz_0,hz_1,h^kv)},
\end{split}
\]
where $S^3_k:=\{ z=(z_0,z_1)\in \C^2\colon \norm{z}^2=k \}$. Namely, $(M,\omega)$ is the one-dimensional complex projective space $\C P^1$ with $k\omega_{FS}$, where $\omega_{FS}$ is the Fubini-Study form which represents the generator of $H^\bullet (\C P^1;\Z)$, and $L$ is the $k$th tensor power of the hyperplane line bundle $H^{\otimes k}$.  

Take and fix an integer $m$. Let us consider the toric $S^1$-action on $M$ and its lift on $L$ which is defined by
\[
g[z_0:z_1,v]:=[z_0:gz_1,g^mv]
\]
for $g\in S^1$ and $[z_0:z_1,v]\in L$. In this example we have the following exactly $k+1$ non $L$-acyclic orbits
\[
\CO_i:=\{[z_0:z_1]\in M\colon \abs{z_1}^2=i\} \ (i=0, 1, \ldots ,k). 
\]
In fact, for an orbit $\CO$ take and fix an element $[z_0:z_1]\in \CO$. Then, $\CO$ can be written as $\CO=\{[z_0:hz_1]\colon h\in S^1 \}$. Suppose $s\in H^0(\CO;(L,\nabla^L)|_{\CO})$ is a non-trivial global parallel section. Then, it is easy to show that $s$ should be of the form
\begin{equation}\label{s}
s([z_0:hz_1])=[z_0:hz_1,h^{\abs{z_1}^2}s_0]
\end{equation}
for some complex number $s_0\in \C$. In particular, by \eqref{s}, $\abs{z_1}^2$ should be integer since $s$ is a global section on $\CO$. Conversely, suppose $\abs{z_1}^2$ is an integer. Then, \eqref{s} defines a non-trivial global parallel section on $\CO$. 
\end{exa}

In~\cite{FFY2,FFY3} we obtained the following theorem. %For more details, see~\cite{FFY2,FFY3}. 
\begin{thm}[\cite{FFY2,FFY3}]\label{main}
Let $(M,\omega)$ be a possibly non-compact symplectic manifold with effective Hamiltonian $S^1$-action and $(L,\nabla^L)\to (M,\omega)$ an $S^1$-equivariant prequantum line bundle on it. Let $V\subset M$ be an $S^1$-invariant open set which contains only $L$-acyclic orbits and whose complement $M\setminus V$ is compact. For these data, 
%For each $S^1$-invariant open set $V$ which contains only $L$-acyclic orbits and whose complement $M\setminus V$ is compact, 
there exists an element $\ind_{S^1}(M,V;L)\in R(S^1)$ of the representation ring $R(S^1)$ of $S^1$ that satisfies the following properties:
\begin{enumerate}
\item $\ind_{S^1}(M,V;L)$ is invariant under continuous deformation of the data. 
\item If $M$ is closed, then, $\ind_{S^1}(M,V;L)$ is equal to the equivariant index of a Spin${}^c$ Dirac operator. 
\item If $V'$ is an $S^1$-invariant open subset of $V$ with complement $M\setminus V'$ compact, then we have 
\[
\ind_{S^1}(M,V;L)=\ind_{S^1}(M,V';L). 
\]
\item If $M'$ is an $S^1$-invariant open neighborhood of $M\setminus V$, then $\ind_{S^1}(M,V;L)$ has the following excision property
\[
\ind_{S^1}(M,V;L)=\ind_{S^1}(M',V\cap M';L|_{M'}). 
\]
\item \label{disjoint union}If $M$ is a disjoint union $M=M_1\coprod M_2$, %such that for $i=1,2$ $V\cap M_i$ is non-empty, or $M_i$ is compact. 
then we have the following sum formula
\[
\ind_{S^1}(M,V;L)=\ind_{S^1}(M_1,V\cap M_1;L|_{M_1})\oplus \ind_{S^1}(M_2,V\cap M_2;L|_{M_2}). 
\]
\item We have a product formula for $\ind_{S^1}(M,V;L)$. For the precise statement see \cite[Theorem~5.8]{FFY2}. 
\end{enumerate}
We call $\ind_{S^1}(M,V;L)$ an {\it equivariant local index}. 
\end{thm}
\begin{rem}
An orbifold version is available. It will be necessary in Section 3. 
\end{rem}

See~\cite{FFY2} for a proof. Let us briefly recall the construction of $\ind_{S^1}(M,V;L)$. The idea used here is the following infinite dimensional analog of the Witten deformation. Let $D\colon \Gamma(W)\to \Gamma(W)$ be the $S^1$-invariant Spin${}^c$ Dirac operator. For $t\ge 0$ consider the following perturbation of $D$
\[
D_t:=D+t\rho D_\text{fiber},
\]
where $\rho$ is an $S^1$-invariant cut-off function on $M$ with $\rho|_{M\setminus V}\equiv 0$ and $\rho\equiv 1$ outside a compact neighborhood of $M\setminus V$, and $D_\text{fiber}$ is an $S^1$-invariant de Rham operator on $V$ along orbits in the following sense, namely, 
\begin{enumerate}
\item $D_\text{fiber}\colon \Gamma\left(W|_V\right)\to \Gamma\left(W|_V\right)$ is an order-one, formally self-adjoint $S^1$-invariant differential operator of degree-one.  
\item $D_\text{fiber}$ contains only derivatives along orbits. 
\item For each orbit $\CO$ in $V$ $D_\text{fiber} |_{\CO}$ is the de Rham operator of $\CO$ with coefficients in $(L,\nabla^L )|_{\CO}$. 
\item For each orbit $\CO$ in $V$ let $u\in \Gamma (TV|_{\CO})$ be an $S^1$-invariant section perpendicular to the orbit direction. $u$ acts on $\Gamma\left(W|_{\CO}\right)$ as the Clifford multiplication $c(u)$. Then, $D_\text{fiber}$ anti-commutes with $c(u)$.%, namely, 
%\[
%D_\text{fiber}\circ c(u)+c(u)\circ D_\text{fiber}=0. 
%\]
%
%condition for $u$: 
%$g(u,X_\xi)=0$ for all $\xi\in \t:=\Lie(S^1)$
%$d\varphi_g\circ u=u\circ \varphi_g$ for all $g\in S^1$
%
%
\end{enumerate}
It is possible to take such a $D_\text{fiber}$. Since $V$ contains only $L$-acyclic orbits the third condition for $D_\text{fiber}$ and Lemma~\ref{key} imply that the kernel of $D_\text{fiber}|_{\CO}$ is trivial for any orbit $\CO$ in $V$.  

First we give a definition of $\ind_{S^1}(M,V;L)$ for the special case where $M$ has a cylindrical end. In \cite{FFY,FFY2,FFY3} we showed the following proposition. 
\begin{prp}\label{Euclidian end}
Under the assumption in Theorem~\ref{main} suppose that $M$ has a cylindrical end $V= N\times (0,\infty)$ and all the data are translationally invariant on the end. Then for a sufficiently large $t\gg 0$, the space of $L^2$-solutions of $D_ts=0$ is finite dimensional and its super-dimension is independent of a sufficiently large $t\gg 0$ and any other continuous deformations of data. 
\end{prp}

\begin{dfn}\label{ind(M,V,W) for cylindrical end}
In the case of Proposition~\ref{Euclidian end} we define the $\ind(M,V,W)$ to be the super-dimension of the space of $L^2$-solutions of $D_ts=0$, namely, 
\[
\ind_{S^1}(M,V;L):=\dim \ker D_t^0\cap L^2 -\dim \ker D_t^1\cap L^2
\]
 for a sufficiently large $t\gg 0$. 
\end{dfn}

For the general end case, we replace $V$ by a cylindrical end so that all the data are translationally invariant on the end, and come down to the cylindrical end case. 
\begin{figure}[hbtp]
\begin{center}
%\hspace*{-5cm}
\input{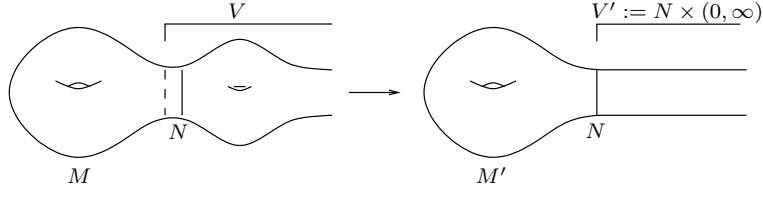}
\label{fig1}
\caption{Replacing $V$ by a cylindrical end}
\end{center}
\end{figure}
We can show that $\ind_{S^1}(M,V,W)$ is well-defined, namely, it does not depend on various choices of the construction. See \cite{FFY,FFY2,FFY3} for more details. 
\begin{rem}
To obtain a product formula we need to formulate and define $\ind_{S^1}(M,V;L)$ for a manifold whose end is the total space of 
a fiber bundle such that both of its base space and its fiber are manifolds with cylindrical end. 
\end{rem}

Let $(L,\nabla^L)\to (M,\omega)$ and $V$ be the data as in Theorem~\ref{main}. Suppose that there exist finitely many mutually disjoint $S^1$-invariant open sets $V_1$, $\ldots$, $V_n$ of $M$ such that $V_1$, $\ldots$, $V_n$, and $V$ form an open covering of $M$, namely, $M=V\cup \left(\cup_{i=1}^nV_i\right)$. Then, for each $i=1, \ldots , n$ the equivariant local index $\ind_{S^1}\left(V_i,V_i\cap V;L|_{V_i}\right)\in R(S^1)$ is defined, and as a corollary of Theorem~\ref{main} we have the following localization formula for $\ind_{S^1}\left(M,V;L\right)$. 
\begin{crl}\label{equivariant localization}
%Let $(L,\nabla^L)\to (M,\omega)$, $V$, and 
$\ind_{S^1}(M,V;L)$ is written as the sum of $\ind_{S^1}\left(V_i,V_i\cap V;L|_{V_i}\right)$'s, namely,
\begin{equation*}%\label{equivariant localization formula}
\ind_{S^1}\left(M,V;L\right)=\bigoplus_{i=1}^n\ind_{S^1}\left(V_i,V_i\cap V;L|_{V_i}\right) . 
\end{equation*}
\end{crl}
Corollary~\ref{equivariant localization} implies that $\ind_{S^1}\left(M,V;L\right)$ can be described in terms of the data restricted to the neighborhood $V_i$ of non $L$-acyclic orbits. 
\begin{proof}[Proof of Corollary~\ref{equivariant localization}]
Since $\cup_{i=1}^nV_i$ is an $S^1$-invariant open neighborhood of $M\setminus V$ the excision property shows
\begin{equation*}%\label{}
\ind_{S^1}(M,V;L)=\ind_{S^1}\left(\cup_{i=1}^nV_i,\cup_{i=1}^n\cap V;L|_{\cup_{i=1}^nV_i}\right)  .
\end{equation*}
Moreover, since $V_i$'s are mutually disjoint, by the sum formula, we obtain the equality in Corollary~\ref{equivariant localization}.
\end{proof}

\begin{rem}
We can prove this theorem in the torus action case. In that case we need to construct an additional geometric structure named \lq\lq strongly acyclic compatible system" on $V$. See~\cite{FFY2,FFY3}. As an application of the theorem for the torus action we can obtain the Danilov formula for a nonsingular projective toric variety $M$~\cite{Dan}. 
%%\begin{thm}[Danilov '78]
%Let $M$ be a nonsingular projective toric variety with $\dim_\C M=n$ and $L\to M$ a $(S^1)^n$-equivariant ample line bundle. Since the $(S^1)^n$-action preserves all the data $(S^1)^n$ acts on $H^0(M;\CO_L)$ by pull-back. Then, $H^0(M;\CO_L)$ has the following irreducible decomposition
%\begin{equation*}
%H^0(M;\CO_L)=\sum_{\gamma^* \in \Delta\cap \t_{\Z}^*}\C_{\gamma^*} ,
%\end{equation*}
%where $\Delta$ is the convex polytope associated to $M$ and $\C_{\gamma^*}$ is the irreducible representation with weight $\gamma^*$. 
It will be explained in \cite{FY5}. 
\end{rem}

\begin{exa}[Equivariant localization formula for $\C P^1$]
Let us consider the case of Example~\ref{non L-acyclic orbits in CP^1}. %In this case the moment map $\mu\colon M\to \t^*\cong \R$ determined by the Kostant formula~\eqref{Kostant formula} is 
%\begin{equation}\label{moment map CP^1}
%\mu([z_0:z_1])=\abs{z_1}^2-m. 
%\end{equation}
Recall that we have the exactly $k+1$ non $L$-acyclic orbits $\CO_0$, $\ldots$, $\CO_k$. 
%\[
%\CO_i:=\{[z_0:z_1]\in M\colon \abs{z_1}^2=i\} \ (i=0, 1, \ldots ,k). 
%\]
For each $i=0,1,\ldots ,k$ we take a sufficiently small positive real number $\varepsilon_i>0$ and define $V_i$ by
\[
V_i:=\{ [z_0:z_1]\in M\colon i-\varepsilon_i<\abs{z_1}^2<i+\varepsilon_i \} . 
\]
We put 
\[
V:=\{[z_0:z_1]\in M\colon \abs{z_1}^2\not\in \Z\}.
\]
Then, for each $i=0,1,\ldots ,k$ the local index $\ind_{S^1}(V_i,V_i\cap V;L|_{V_i})$ is defined. Now we show the following formula
\begin{equation}\label{computation}
\ind_{S^1}(V_i,V_i\cap V;L|_{V_i})=\C_{i-m}. 
\end{equation}
%{\color{red}
%Recall the definition of $\ind_{S^1}(V_i,V_i\cap V;L|_{V_i})$. We deform the end $V_i\cap V$ of $V_i$ cylindrically so that all the data are translationally invariant, and define 
%%$\ind_{S^1}(V_i,V_i\cap V;L|_{V_i})$ is defined as
%\[
%\ind_{S^1}(V_i,V_i\cap V;L|_{V_i}):=\ker D_t|_{\wedge^{0,even}T^*V_i\otimes L|_{V_i}}\cap L^2-\ker D_t|_{\wedge^{0,odd}T^*V_i\otimes L|_{V_i}}\cap L^2
%\]
%for a sufficiently large $t>0$, where $D_t$ is the perturbed Dirac operator as in {\color{red}reference number}. 
%}
For each $i=1,\ldots ,k-1$, $(L,\nabla^L)|_{V_i}\to (V_i,\omega|_{V_i})$ is equivariantly isomorphic to the trivial line bundle on the cylinder $S^1\times (i-\varepsilon_i ,i+\varepsilon_i)$% with connection
\[
\left(S^1\times (i-\varepsilon_i ,i+\varepsilon_i)\times \C ,d-2\pi\sqrt{-1}rd\theta \right)\to \left(S^1\times (i-\varepsilon_i ,i+\varepsilon_i),dr\wedge d\theta\right)
\]
with $S^1$-action
\begin{equation}\label{action1}
g(e^{2\pi\sqrt{-1}\theta},r,v):=(ge^{2\pi\sqrt{-1}\theta},r,g^mv)
\end{equation}
for $(e^{2\pi\sqrt{-1}\theta},r,v)\in S^1\times (i-\varepsilon_i ,i+\varepsilon_i)\times \C$. The isomorphism $f_i\colon S^1\times (i-\varepsilon_i ,i+\varepsilon_i)\times \C\to (L,\nabla^L)|_{V_i}$ is given as 
\[
f_i(e^{2\pi\sqrt{-1}\theta},r,v):=\left[\sqrt{k-r}:e^{2\pi\sqrt{-1}\theta}\sqrt{r},v\right]. 
\]
For each $i=0, k$, $(L,\nabla^L)|_{V_i}\to (V_i,\omega|_{V_i})$ is equivariantly isomorphic to the trivial line bundle on the disc $D_{\varepsilon_i}:=\{ w\in \C \colon \abs{w}^2<\varepsilon_i \}$% with connection
\[
\left(D_{\varepsilon_i}\times \C ,d+\dfrac{1}{2}(wd\bar w-\bar wdw) \right)\to \left(D_{\varepsilon_i},\dfrac{\sqrt{-1}}{2\pi}dw\wedge d\bar w\right)
\]
with $S^1$-action
\begin{equation}\label{action2}
g(w,v):=
\begin{cases}
(gw,g^mv) & \text{if }i=0\\
(g^{-1}w,g^{m-k}v) & \text{if }i=k
\end{cases}
\end{equation}
for $(w,v)\in D_{\varepsilon_i}\times \C$. The isomorphism $f_i\colon D_{\varepsilon_i}\times \C\to (L,\nabla^L)|_{V_i}$ is given as 
\[
f_i(w,v):=
\begin{cases}
[\sqrt{k-\abs{w}^2}:w,v] & \text{if }i=0\\
[w:\sqrt{k-\abs{w}^2},v] & \text{if }i=k. 
\end{cases}
\]
With the above identifications $f_i$, we can compute $\ind_{S^1}(V_i,V_i\cap V;L|_{V_i})$. According to \cite[Remark~6.10]{FFY} both of $\dim \ker D_t^0\cap L^2$ and $\dim\ker D_t^1\cap L^2$ in Definition~\ref{ind(M,V,W) for cylindrical end} for $\ind_{S^1}(V_i,V_i\cap V;L|_{V_i})$ are computed as
\[
\dim \ker D_t^0\cap L^2=1,\ \dim\ker D_t^1\cap L^2=0,
\] 
%Moreover, we can detect 
and a generator of $\ker D_t^0\cap L^2$ is given as 
\begin{equation}\label{generator1}
s_i(e^{2\pi\sqrt{-1}\theta},r)=\left(e^{2\pi\sqrt{-1}\theta},r, a_i(r)e^{2\pi\sqrt{-1}i\theta}\right)
\end{equation}
for $ i=1,\ldots ,k-1$ and 
\begin{equation}\label{generator2}
s_i(w)=\left(w,b_i(\abs{w})\right)
\end{equation}
for $i=0, k$, where $a_i(r)$ and $b_i(\abs{w})$ are some functions on $r$ and $\abs{w}$, respectively. See \cite[Remark~6.10]{FFY},or \cite[Section~5.3]{Y7} for more details. The $S^1$ acts on $\ker D_t^0\cap L^2$ by pull-back. For $g\in S^1$ we denote by $\varphi_g$ and $\psi_g$ the $S^1$-action on $M$ and $L$, respectively. By using the explicit expressions for the $S^1$-actions~\eqref{action1}, \eqref{action2} and the generators~\eqref{generator1}, \eqref{generator2} the $S^1$-actions on the generators are written as 
\[
(\psi_{g^{-1}}\circ s_i\circ \varphi_g)(e^{2\pi\sqrt{-1}\theta},r)=g^{i-m}s_i(e^{2\pi\sqrt{-1}\theta},r)
\]
for $i=1,\ldots ,k-1$, and 
\[
(\psi_{g^{-1}}\circ s_i\circ \varphi_g)(w)=
\begin{cases}
g^{-m}\left(s_i(w)\right) & \text{if }i=0\\
g^{k-m}\left(s_i(w)\right) & \text{if }i=k. 
\end{cases}
\]
Hence, we can obtain the formula~\eqref{computation}. 

%By~\eqref{moment map CP^1} the image of the moment map is $[-m,k-m]$. Together with it, 
By the second property for $\ind_{S^1}(M,V;L)$ in Theorem~\ref{main} and the formula~\eqref{computation} for $\ind_{S^1}(V_i,V_i\cap V;L|_{V_i})$, the equivariant localization formula in Corollary~\ref{equivariant localization} is written as
\[
\begin{split}
H^0(M;\CO_L)&=\ind_{S^1} D\\
&=\ind_{S^1}(M,V;L)\\
&=\bigoplus_{i=0}^k\ind_{S^1}(V_i,V_i\cap V;L|_{V_i})\\
& = \bigoplus_{i=0}^k \C_{i-m}.%\\
%&=\bigoplus_{\gamma^*\in \mu(M)\cap \t^*_\Z}\C_{\gamma^*}. 
\end{split}
\]
%This is Danilov's formula for $\C P^1$. 
\end{exa}

\subsection{$\ind_{S^1}^{\gamma^*}(M,V;L)$}
Let $\t^*_\Z$ be the weight lattice of $S^1$. For each $\gamma^*\in \t^*_\Z$ and an element $U\in R(S^1)$ let us denote by $U^{\gamma^*}$ the multiplicity of the irreducible representation with weight $\gamma^*$ in $U$. By taking the multiplicities of the irreducible representations with weight $\gamma^*$ in the both sides of the equivariant localization formula in Corollary~\ref{equivariant localization}, we obtain the following localization formula for $\ind_{S^1}(M,V;L)^{\gamma^*}$.
\begin{equation}\label{localization formula for multiplicity}
\ind_{S^1}\left(M,V;L\right)^{\gamma^*}=\bigoplus_{i=1}^n\ind_{S^1}\left(V_i,V_i\cap V;L|_{V_i}\right)^{\gamma^*} . 
\end{equation}

In this subsection, for each $\gamma^*\in \t^*_\Z$, we define an $(L,\gamma^*)$-acyclic condition which is a milder condition than the $L$-acyclic condition. By using the $(L,\gamma^*)$-acyclic condition we obtain a version of a local index, which is denoted by $\ind_{S^1}^{\gamma^*}(M,O;L)$, and its localization formula. In particular, \eqref{localization formula for multiplicity} is obtained as a special case of the localization formula for $\ind_{S^1}^{\gamma^*}(M,O;L)$. 

Since the $S^1$-action preserves all the data, for each orbit $\CO$, $S^1$ acts on $H^0 (\CO; (L,\nabla^L )|_{\CO})$ by pull-back. 
%We denote by $H^0 (\CO; (L,\nabla^L )|_{\CO})^{\gamma^*}$ the irreducible component with weight $\gamma^*$ of $H^0 (\CO; (L,\nabla^L )|_{\CO})$ as a $S^1$-representation. 
\begin{dfn}%[$(L,\gamma^*)$-acyclic orbit]
For each $\gamma^*\in \t^*_\Z$ an orbit $\CO$ is said to be $(L,\gamma^*)$-acyclic if $\CO$ does not consist of a fixed point and satisfies the condition $H^0 (\CO; (L,\nabla^L )|_{\CO})^{\gamma^*}=0$. 
\end{dfn}
\begin{rem}
By definition, any $L$-acyclic orbit are $(L,\gamma^*)$-acyclic orbit. 
\end{rem}

The following lemma is a version of Lemma~\ref{key} for $(L,\gamma^*)$-acyclic orbits. 
\begin{lmm}\label{key2}
Let $\CO$ be an orbit of the $S^1$-action on $M$. Then, the following conditions are equivalent:
\begin{enumerate}
\item $\CO$ is $(L,\gamma^*)$-acyclic. 
%\item $(L,\nabla^L)|_{\CO}$ admits no non-trivial global parallel section. 
\item $H^\bullet \left(\CO;(L,\nabla^L)|_{\CO}\right)^{\gamma^*}=0$. 
\item The irreducible component with weight $\gamma^*$ of the kernel of the de Rham operator of $\CO$ with coefficients in $L$ vanishes. 
\end{enumerate}
\end{lmm}
The proof is similar to that of Lemma~\ref{key}. 

\begin{exa}[Non $(L,\gamma^*)$-acyclic orbits in $\C P^1$]\label{non (L gamma)-acyclic orbits in CP^1}
Let us find non $(L,\gamma^*)$-acyclic orbits for Example~\ref{non L-acyclic orbits in CP^1}. By definition, non $(L,\gamma ^*)$-acyclic orbits are orbits consisting of a fixed point, or orbits with $H^0 (\CO; (L,\nabla^L )|_{\CO})^{\gamma^*}\neq 0$. The orbits of the former type are 
\[
\CO_0=\{ [z_0:0] \},\ \CO_k=\{ [0:z_1]\}.%\ \CO_{m+\gamma^*}=\{[z_0:z_1]\in M\colon \abs{z_1}^2=m+\gamma^*\} .
\]
We can show that there exists an orbit of the latter type if and only if $0\le m+\gamma^*\le k$, and in that case we have the unique orbit of the latter type which is 
\[
%\CO_0=\{ [z_0:0] \},\ \CO_k=\{ [0:z_1]\},\ 
\CO_{m+\gamma^*}=\{[z_0:z_1]\in M\colon \abs{z_1}^2=m+\gamma^*\} .
\]
%which satisfies  $H^0 (\CO; (L,\nabla^L )|_{\CO})^{\gamma^*}\neq 0$ as follows. 
Recall that $\CO_i$s are the only orbits which satisfy $H^0 (\CO_i; (L,\nabla^L )|_{\CO_i})\neq 0$, and in that case an element $s\in H^0 (\CO_i; (L,\nabla^L )|_{\CO_i})$ has the form~\eqref{s}. $S^1$ acts on $H^0 (\CO_i; (L,\nabla^L )|_{\CO_i})$ by pull-back. For $g\in S^1$ and $s([z_0:hz_1])=[z_0:hz_1,h^{\abs{z_1}^2}s_0]\in H^0(\CO_i;(L,\nabla^L )|_{\CO_i})$ the $S^1$-action can be written as
\[
(\psi_{g^{-1}}\circ s\circ \varphi_g)([z_0:hz_1])=[z_0:hz_1,g^{\abs{z_1}^2-m}h^{\abs{z_1}^2}s_0]. 
\]
Thus, $(\psi_{g^{-1}}\circ s\circ \varphi_g)=g^{\gamma^*}s$ if and only if $\abs{z_1}^2-m=\gamma^*$.  
\end{exa}

Now we have a version of Theorem~\ref{main}. See~\cite{FFY3} in case of $\gamma^*=0$. 
\begin{thm}\label{main2}
Let $(M,\omega)$ be a possibly non-compact symplectic manifold with effective Hamiltonian $S^1$-action and $(L,\nabla^L)\to (M,\omega)$ an $S^1$-equivariant prequantum line bundle on it. Let $O\subset M$ be an $S^1$-invariant open set which contains only $(L,\gamma^*)$-acyclic orbits and whose complement $M\setminus O$ is compact. For these data, there exists an integer $\ind_{S^1}^{\gamma^*}(M,O;L)\in \Z$ that satisfies the same properties as in Theorem~\ref{main}. 
%We call $\ind_{S^1}(M,O;L)^{\gamma^*}$ an {\it equivariant local index}. 
\end{thm}

\begin{rem}\label{remark for multiplicity}
In order to define $\ind_{S^1}^{\gamma^*}(M,O;L)\in \Z$ we replace the $L$-acyclic condition by the $(L,\gamma^*)$-acyclic condition in the construction of the equivariant local index, and consider the multiplicity of the irreducible representation with weight $\gamma^*$ in $\ker D_t^0\cap L^2 -\ker D_t^1\cap L^2$ for the perturbed Spin${}^c$ Dirac operator $D_t$ instead of $\ker D_t^0\cap L^2 -\ker D_t^1\cap L^2$ itself. In particular, since $L$-acyclic orbits are $(L,\gamma^*)$-acyclic, $V$ in Theorem~\ref{main} can be taken as $O$ in Theorem~\ref{main2}. In this case, by definition, $\ind_{S^1}^{\gamma^*}(M,O;L)$ is equal to $\ind_{S^1}(M,V;L)^{\gamma^*}$. 
\end{rem}

Let $(L,\nabla^L)\to (M,\omega)$ and $O$ be the data as in Theorem~\ref{main2}. Suppose that there exist finitely many mutually disjoint $S^1$-invariant open sets $O_1$, $\ldots$, $O_l$ of $M$ such that $O_1$, $\ldots$, $O_l$, and $O$ form an open covering of $M$, namely, $M=O\cup \left(\cup_{i=1}^lO_i\right)$. Then, for each $i=1, \ldots , l$ $\ind_{S^1}^{\gamma^*}\left(O_i,O_i\cap O;L|_{O_i}\right)\in \Z$ is well defined, and we have the following localization formula for $\ind_{S^1}^{\gamma^*}\left(M,O;L\right)$. 
\begin{crl}\label{localization for multiplicities}
$\ind_{S^1}^{\gamma^*}(M,O;L)$ is written as the sum of $\ind_{S^1}^{\gamma^*}\left(O_i,O_i\cap O;L|_{O_i}\right)$'s, namely,
\begin{equation*}%\label{localization formula for multiplicity 2}
\ind_{S^1}^{\gamma^*}\left(M,O;L\right)=\bigoplus_{i=1}^l\ind_{S^1}^{\gamma^*}\left(O_i,O_i\cap O;L|_{O_i}\right) . 
\end{equation*}
\end{crl}
This formula implies that $\ind_{S^1}^{\gamma^*}\left(M,O;L\right)$ can be described in terms of the data restricted to a sufficiently neighborhood of the fixed point set and orbits with $H^\bullet \left(\CO;(L,\nabla^L)|_{\CO}\right)^{\gamma^*}\neq 0$. 

\begin{rem}
By Remark~\ref{remark for multiplicity}, if we take $V$ and $V_i$'s in Corollary~\ref{equivariant localization} as $O$ and $O_i$'s in Corollary~\ref{localization for multiplicities}, respectively, then %$\ind_{S^1}^{\gamma^*}(M,V;L)$ and $\ind_{S^1}^{\gamma^*}(V_i,V_i\cap V;L|_{V_i})$'s are equal to $\ind_{S^1}(M,V;L)^{\gamma^*}$ and $\ind_{S^1}(V_i,V_i\cap V;L|_{V_i})^{\gamma^*}$'s, respectively. In particular, 
\eqref{localization formula for multiplicity} is obtained by Corollary~\ref{localization for multiplicities}. 
\end{rem}

\begin{exa}[Localization formula for multiplicities in $\C P^1$]
In Example~\ref{non (L gamma)-acyclic orbits in CP^1} we showed that for each $\gamma^*\in \t^*_\Z$ with $0< m+\gamma^*< k$ there are exactly three non $(L,\gamma^*)$-acyclic orbits $\CO_0$, $\CO_k$, and $\CO_{m+\gamma^*}$, otherwise we have exactly two non $(L,\gamma^*)$-acyclic orbits $\CO_0$ and $\CO_k$. We put 
\[
O_0:=V_0,\ O_k:=V_k,\ O_{m+\gamma^*}:=V_{m+\gamma^*},\ \text{and } O:=\{ [z_0:z_1]\in M\colon \abs{z_1}^2\neq 0, k, m+\gamma^*\} .
\]
%
%For each $\gamma^*\in \t^*_\Z$ with $m+\gamma^*\le 0$ or $k\le m+\gamma^*$ we showed in Example~\ref{non (L gamma)-acyclic orbits in CP^1} that there are exactly two non $(L,\gamma^*)$-acyclic orbits, $\CO_0$, $\CO_k$. In that case we put 
%\[
%O_0:=V_0,\ O_1:=V_k,\ \text{and } O:=\{ [z_0:z_1]\in M\colon \abs{z_1}^2\neq 0, k\} .
%\]
Then, for each $i$ $\ind_{S^1}^{\gamma^*}(O_i,O_i\cap O;L|_{O_i})$ is defined. By definition, $(O_i,O_i\cap O)$ is equal to $(V_i,V_i\cap V)$. Hence, by Remark~\ref{remark for multiplicity} and the formula~\eqref{computation} we obtain 
%\begin{equation}
\begin{align}
\ind_{S^1}^{\gamma^*}(O_i,O_i\cap O;L|_{O_i})&=\ind_{S^1}^{\gamma^*}(V_i,V_i\cap V;L|_{V_i})\nonumber\\
&=\ind_{S^1}(V_i,V_i\cap V;L|_{V_i})^{\gamma^*}\nonumber\\
&=
\begin{cases}
1 & \text{if }0\le m+\gamma^*\le k \text{ and } i={m+\gamma^*} \\ 
0 & \text{otherwise.}
\end{cases}\nonumber%\label{computation2}
\end{align}
\end{exa}

\section{A special case}
Let $(L,\nabla^L)\to (M,\omega)$ be as above. For $g\in S^1$ we denote by $\varphi_g$ and $\psi_g$ the $S^1$-action on $M$ and $L$, respectively. It is well known that corresponding to the infinitesimal lift of the $S^1$-action on $M$ to $L$, the moment map $\mu\colon M\to \t^*$ is determined uniquely by the following Kostant formula 
\begin{equation}\label{Kostant formula}
\dfrac{d}{dt}\Big|_{t=0}\psi_{e^{-t\xi}}\circ s \circ \varphi_{e^{t\xi}}=\nabla_{X_\xi}s+2\sqrt{-1}\pi\< \mu ,\xi\>s
\end{equation}
for $\xi\in \t$ and $s\in \Gamma (L)$, where $\t$ is the Lie algebra of $S^1$, $\< \ ,\ \>$ is the natural pairing between $\t^*$ and $\t$, and $X_\xi\in \Gamma (TM)$ is the infinitesimal action of $\xi$. We have the following relationship between non $L$-acyclic orbits, $(L,\gamma^*)$-acyclic orbits, and the values of $\mu$. 
\begin{lmm}\label{L-acyclic, (L,gamma)-acyclic orbits and moment map}
%\begin{enumerate}
%\item 
\textup{(1)} Non $L$-acyclic orbits are contained in $\mu^{-1}(\t^*_\Z)$. In particular, fixed points are contained in $\mu^{-1}(\t^*_\Z)$. \\
%\item 
\textup{(2)} Orbits with $H^\bullet \left(\CO;(L,\nabla^L)|_{\CO}\right)^{\gamma^*}\neq 0$ are contained in $\mu^{-1}(\gamma^*)$.
%\end{enumerate}
\end{lmm}
\begin{proof}
Let $\CO$ be a non $L$-acyclic orbit with $\mu(\CO)=\eta^*\in \t^*$. Then, by definition, there exists a non-trivial global parallel section $s\in H^0\left(\CO;(L,\nabla^L)|_{\CO}\right)$. For any element $\xi$ in the integral lattice $\t_\Z$ we put 
\[
s_t:=\psi_{e^{-t\xi}}\circ s \circ \varphi_{e^{t\xi}}. 
\]
By \eqref{Kostant formula}, we have
\[
\dfrac{d}{dt}s_t(x)=2\sqrt{-1}\pi\< \eta^*,\xi\>s_t(x) 
\]
for $x\in \CO$. Then, $s_t$ has the form
\[
s_t=e^{2\sqrt{-1}\pi\< \eta^*,\xi\>t}s.
\]
Since $\xi \in \t_\Z$, by putting $t=1$, 
\[
s=s_1=e^{2\sqrt{-1}\pi\< \eta^*,\xi\>}s. 
\]
Thus, $\<\eta^*,\xi\>$ should be integer for arbitrary $\xi \in \t_\Z$. This implies the first part. 

Let $\CO$ be an orbit with $H^\bullet \left(\CO;(L,\nabla^L)|_{\CO}\right)^{\gamma^*}\neq 0$. Then, there exists a non-trivial global parallel section $s\in H^0\left(\CO;(L,\nabla^L)|_{\CO}\right)^{\gamma^*}$. For any element $\xi\in \t$, by \eqref{Kostant formula}, we have
\[
\begin{split}
2\pi\sqrt{-1}\< \gamma^* ,\xi\>s(x)&=\dfrac{d}{dt}\Big|_{t=0}\psi_{e^{-t\xi}}\circ s \circ \varphi_{e^{t\xi}}(x)\\
&=2\pi\sqrt{-1}\< \mu(x),\xi\>s(x) 
\end{split}
\]
for $x\in \CO$. Since $s$ is non-trivial this implies the second part. 
\end{proof}

In the rest of this section we assume that $\mu$ is proper and the cardinality of $\mu(M)\cap \t^*_\Z$ is finite. For each $\gamma^*\in \mu(M)\cap \t^*_\Z$ let $V_{\gamma^*}$ be a sufficiently small $S^1$-invariant neighborhood of $\mu^{-1}(\gamma^*)$ so that $\{V_{\gamma^*}\}_{\gamma^*\in \mu(M)\cap \t^*_\Z}$ are mutually disjoint. Let $V$ be the complement of $\mu^{-1}(\t^*_\Z)$, namely, $V:=M\setminus \mu^{-1}(\t^*_\Z)$. By Lemma~\ref{L-acyclic, (L,gamma)-acyclic orbits and moment map} $V$ contains only $L$-acyclic orbits. Moreover, by assumption, $V_{\gamma^*}\setminus V$ is compact for each $\gamma^*\in \mu(M)\cap \t^*_\Z$. Hence, for each $\gamma^*\in \mu(M)\cap \t^*_\Z$ the equivariant local index $\ind_{S^1}\left(V_{\gamma^*},V_{\gamma^*}\cap V;L|_{V_{\gamma^*}}\right)\in R(S^1)$ is defined. By applying Corollary~\ref{equivariant localization} to this case we have the following localization formula for $\ind_{S^1}\left(M,V;L\right)$. 
\begin{equation}\label{equivariant localization formula 2}
\ind_{S^1}\left(M,V;L\right)=\bigoplus_{\gamma^*\in \mu(M)\cap \t^*_\Z}\ind_{S^1}\left(V_{\gamma^*},V_{\gamma^*}\cap V;L|_{V_{\gamma^*}}\right) . 
\end{equation}

We show the following theorem. 
\begin{thm}\label{vanishing formula for equivariant local index}
For each $\gamma^*\in \mu(M)\cap \t^*_\Z$ and $\sigma^*\in \t^*_\Z$ with $\gamma^*\neq \sigma^*$
\begin{equation*}%\label{vanishing for equivariant local index}
\ind_{S^1}\left(V_{\gamma^*},V_{\gamma^*}\cap V;L|_{V_{\gamma^*}}\right)^{\sigma^*}=0.
%\begin{cases}
%\ind_{S^1}\left(M,V;L\right)^{\sigma^*} & \text{if }\gamma^*=\sigma^*, \\
%0 & \text{if }\gamma^*\neq \sigma^*.
%\end{cases}
\end{equation*}
%In particular, by~\eqref{equivariant localization formula 2} and ~\eqref{multiplicity for equivariant local index}, $\ind_{S^1}\left(V_{\gamma^*},V_{\gamma^*}\cap V;L|_{V_{\gamma^*}}\right)$ is the irreducible component with weight $\gamma^*$ in $\ind_{S^1}\left(M,V;L\right)$. 
\end{thm}
\begin{proof}
%Since the formula~\eqref{multiplicity for equivariant local index} for $\gamma^*=\sigma^*$ is obtained from ~\eqref{equivariant localization formula 2} and ~\eqref{multiplicity for equivariant local index} for $\gamma^*\neq \sigma^*$, it is sufficient to prove the formula~\eqref{multiplicity for equivariant local index} for $\gamma^*\neq \sigma^*$. 
%
Since $L$-acyclic orbits are $(L,\sigma^*)$-acyclic $\ind_{S^1}^{\sigma^*}(V_{\gamma^*},V_{\gamma^*}\cap V;L|_{V_{\gamma^*}})$ is also defined, and by Remark~\ref{remark for multiplicity} it is equal to $\ind_{S^1}\left(V_{\gamma^*},V_{\gamma^*}\cap V;L|_{V_{\gamma^*}}\right)^{\sigma^*}$. %and, by definition, $\ind_{S^1}\left(V_{\gamma^*},V_{\gamma^*}\cap V;L|_{V_{\gamma^*}}\right)^{\sigma^*}$ is equal to it. 
%By applying Corollary~\ref{localization for multiplicities} to 
$\ind_{S^1}^{\sigma^*}(V_{\gamma^*},V_{\gamma^*}\cap V;L|_{V_{\gamma^*}})$ is described in terms of the data restricted to a sufficiently neighborhood of the fixed point set and orbits with $H^\bullet \left(\CO;(L,\nabla^L)|_{\CO}\right)^{\sigma^*}\neq 0$. By Lemma~\ref{L-acyclic, (L,gamma)-acyclic orbits and moment map} and the definition of $V_{\gamma^*}$, if $\gamma^*\neq \sigma^*$, then $V_{\gamma^*}$ contains no orbits of the latter type. 

Suppose $V_{\gamma^*}$ contains fixed points. By Lemma~\ref{L-acyclic, (L,gamma)-acyclic orbits and moment map} the fixed point set $\left(V_{\gamma^*}\right)^{S^1}$ is contained in $\mu^{-1}(\gamma^*)$. In particular, $\left(V_{\gamma^*}\right)^{S^1}$ is compact since $\mu$ is proper. Suppose $\left(V_{\gamma^*}\right)^{S^1}$ has the exactly $l$ connected components $\left(V_{\gamma^*}\right)^{S^1}_1$, $\ldots$ , $\left(V_{\gamma^*}\right)^{S^1}_l$. For each $i=1, \ldots ,l$ we take a sufficiently small $S^1$-invariant neighborhood $O_i$ of $\left(V_{\gamma^*}\right)^{S^1}_i$, and also put $O:=V_{\gamma^*}\setminus \left(V_{\gamma^*}\right)^{S^1}$. Then, for each $i=1,\ldots ,l$ $\ind_{S^1}^{\sigma^*}(O_i,O_i\cap O;L|_{O_i})$ is defined, and by the third property in Theorem~\ref{main2} and Corollary~\ref{localization for multiplicities} for $\ind_{S^1}^{\sigma^*}\left(V_{\gamma^*},O;L|_{V_{\gamma^*}}\right)$  we have 
\[
\begin{split}
\ind_{S^1}^{\sigma^*}\left(V_{\gamma^*},V_{\gamma^*}\cap V;L|_{V_{\gamma^*}}\right)&=\ind_{S^1}^{\sigma^*}\left(V_{\gamma^*},O;L|_{V_{\gamma^*}}\right)\\
&=\bigoplus_{i=1}^l\ind_{S^1}^{\sigma^*}(O_i,O_i\cap O;L|_{O_i}). 
\end{split}
\]
Now, for each fixed point $x_0\in V_{\gamma^*}$, the fiber $L_{x_0}$ becomes a representation of $S^1$. By the Kostant formula~\eqref{Kostant formula}, for $\xi\in \t$ and $s\in \Gamma (L)$ we have  
\[
\begin{split}
\dfrac{d}{dt}\Big|_{t=0}\psi_{e^{-t\xi}}\left(s(x_0)\right)&=\dfrac{d}{dt}\Big|_{t=0}\psi_{e^{-t\xi}}\circ s \circ \varphi_{e^{t\xi}}(x_0)\\
&=\left(\nabla_{X_\xi}s\right)(x_0)+2\sqrt{-1}\pi\< \mu(x_0) ,\xi\>s(x_0)\\
&=2\sqrt{-1}\pi\< \gamma^*,\xi\>s(x_0). 
\end{split}
\]
This implies $\psi_g(v)=g^{-\gamma^*}v$ for $g\in S^1$ and $v\in L_{x_0}$. By definition it is easy to see that 
\[
\ind_{S^1}^{\sigma^*}(O_i,O_i\cap O;L|_{O_i})=\ind_{S^1}^0\left(O_i,O_i\cap O;L\otimes \C_{\sigma^*}|_{O_i}\right). 
\]
Thus, by \cite[Theorem~4.1]{FFY3}, for $\sigma^*\neq \gamma^*$ we obtain
\[
\ind_{S^1}^{\sigma^*}(O_i,O_i\cap O;L|_{O_i})=0
\]
for each $i=1, \ldots ,l$. This proves the theorem. 
%
%The formula~\eqref{multiplicity for equivariant local index} for $\gamma^*=\sigma^*$ is obtained by comparing the multiplicities of the irreducible components with weight $\sigma^*$ in the both sides of \eqref{equivariant localization formula 2} and the formula~\eqref{multiplicity for equivariant local index} for $\gamma^*\neq \sigma^*$. 
\end{proof}

%Theorem~\ref{multiplicity formula for equivariant local index} implies that \eqref{equivariant localization formula 2} gives an irreducible decomposition. In particular, the irreducible component contained in $\ind_{S^1}\left(M,V;L\right)$ corresponds one-to-one to $\mu(M)\cap \t^*_{\Z}$ and 
As a corollary we have the following formula.
\begin{crl}\label{multiplicity formula for equivariant local index}
\begin{equation*}
\ind_{S^1}(M,V;L)^{\gamma^*}=
\begin{cases}
\ind_{S^1}^{\gamma^*}(V_{\gamma^*},V_{\gamma^*}\cap V;L|_{V_{\gamma^*}}) & \text{if } \gamma^*\in \mu(M)\cap \t^*_{\Z}\\
0 & \text{otherwise}
\end{cases}.
\end{equation*}
\end{crl}

So, we have a natural question: 
\begin{question}\label{question}
How to compute $\ind_{S^1}^{\gamma^*}\left(V_{\gamma^*},V_{\gamma^*}\cap V;L|_{V_{\gamma^*}}\right)$?
\end{question}

We give a partial answer of this question. First let us consider the case where $\gamma^*=0$. Suppose $0\in \t^*_{\Z}$ is a regular value of $\mu$. Then, a new symplectic orbifold $(M_0,\omega_0)$ with prequantum line bundle $(L_0,\nabla^{L_0})$ is obtained by the symplectic reduction, namely, 
\[
%\begin{split}
(M_0,\omega_0 ):=\left(\mu^{-1}(0),\omega|_{\mu^{-1}(0)}\right)/S^1,\ 
(L_0,\nabla^{L_0}):=\left((L,\nabla^L)|_{\mu^{-1}(0)}\right)/S^1.
%\end{split}
\]
Since $\mu$ is proper $M_0$ is compact. Let $D_0$ be the Spin${}^c$ Dirac operator on $(M_0,\omega_0)$ with coefficients in $L_0$. Then, in \cite[Section~5.2]{FFY3} we showed the following formula. 
\begin{thm}[\cite{FFY3}]\label{[Q,R]=0}
Let $\gamma^*=0\in \t^*_{\Z}$ be a regular value of $\mu$. Then, $\ind_{S^1}^0\left(V_0,V_0\cap V;L|_{V_0}\right)$ is equal to the index of $D_0$.%, namely, 
%\begin{equation*}
%\ind_{S^1}^0\left(V_0,V_0\cap V;L|_{V_0}\right)=\ind D_0. 
%\end{equation*}
\end{thm}
\begin{rem}
Note that if $0$ is a regular value of $\mu$, $V_0$ contains no fixed points. In fact, by Lemma~\ref{L-acyclic, (L,gamma)-acyclic orbits and moment map} and the definition of $V_0$, if fixed points exist in $V_0$, they should be contained in $\mu^{-1}(0)$. But, by assumption, $0$ is a regular value of $\mu$. Thus $V_0$ contains no fixed point. In particular, $\ind_{S^1}^0(V_0,V_0\cap V;L|_{V_0})$ is described in terms of the data restricted to a sufficiently neighborhood of the orbits with $H^\bullet \left(\CO;(L,\nabla^L)|_{\CO}\right)^0\neq 0$. 
%\begin{proof}[Outline of the proof]
%
%\end{proof}
\end{rem}

Next let us consider the general case. For a general regular value $\gamma^*\in \t^*_\Z$ of $\mu$ we use the shifting trick. By tensoring $\C_{\gamma^*}$ with $(L,\nabla^L)$ we obtain the prequantum line bundle $(L,\nabla^L)\otimes \C_{\gamma^*}$ on $(M,\omega)$ with shifted $S^1$-action. Then, the moment map associated to the shifted $S^1$-action, which we denote by $\mu_{\gamma^*}$, is written as
\begin{equation}\label{mu_gamma}
\mu_{\gamma^*}=\mu-\gamma^*. 
\end{equation}
Since $\gamma^*$ is a regular value of $\mu$ $0$ is a regular value of $\mu_{\gamma^*}$. Hence, by the symplectic reduction for the shifted $S^1$-action, a new compact symplectic orbifold $(M_{\gamma^*},\omega_{\gamma^*})$ with prequantum line bundle $(L_{\gamma^*},\nabla^{L_{\gamma^*}})$ is obtained as
\[
%\begin{split}
(M_{\gamma^*},\omega_{\gamma^*} ):=\left(\mu^{-1}_{\gamma^*}(0),\omega|_{\mu^{-1}_{\gamma^*}(0)}\right)/S^1,\ 
(L_{\gamma^*},\nabla^{L_{\gamma^*}}):=\left((L,\nabla^L)\otimes \C_{\gamma^*}|_{\mu^{-1}_{{\gamma^*}}(0)}\right)/S^1.
%\end{split}
\]
Let $D_{\gamma^*}$ be the Spin${}^c$ Dirac operator on $(M_{\gamma^*},\omega_{\gamma^*})$ with coefficients in $L_{\gamma^*}$. Then, as a corollary of Theorem~\ref{[Q,R]=0} we obtain the following formula. 
\begin{crl}\label{multiplicity formula for equivariant local index 2}
For a regular value $\gamma^*\in \t^*_{\Z}$ of $\mu$ $\ind_{S^1}^{\gamma^*}\left(V_{\gamma^*},V_{\gamma^*}\cap V;L|_{V_{\gamma^*}}\right)$ is equal to the index of $D_{\gamma^*}$. 
%\begin{equation*}
%\ind_{S^1}^{\gamma^*}\left(V_{\gamma^*},V_{\gamma^*}\cap V;L|_{V_{\gamma^*}}\right)=\ind D_{\gamma^*}. 
%\end{equation*}
\end{crl}
\begin{proof}
By~\eqref{mu_gamma} $V_{\gamma^*}$ is a sufficiently small $S^1$-invariant neighborhood of $\mu_{\gamma^*}^{-1}(0)$. Thus, by Theorem~\ref{[Q,R]=0} for the prequantum line bundle $(L,\nabla^L)\otimes \C_{\gamma^*}$ on $(M,\omega)$ with shifted $S^1$-action we obtain
\begin{equation*}
\ind_{S^1}^0\left(V_{\gamma^*},V_{\gamma^*}\cap V;L\otimes \C_{\gamma^*}|_{V_{\gamma^*}}\right)=\ind D_{\gamma^*}. 
\end{equation*}
On the other hand,  it is easy to see that 
\[
\ind_{S^1}^{\gamma^*}\left(V_{\gamma^*},V_{\gamma^*}\cap V;L|_{V_{\gamma^*}}\right)=\ind_{S^1}^0\left(V_{\gamma^*},V_{\gamma^*}\cap V;L\otimes \C_{\gamma^*}|_{V_{\gamma^*}}\right). 
\]
This proves the corollary. 
\end{proof}

In particular, for a closed $M$, we obtain the the quantization conjecture for the $S^1$-action.
\begin{crl}[\cite{GuS5, Gu2,Duistermaat-Guillemin-Meinrenken-Wu,Mein1,Vergne1,Vergne2,Mein2,Tian-Zhang}, etc.]
Let $(L,\nabla^L)\to (M,\omega)$ be as above. Assume $M$ is closed. If $\gamma^*\in \t^*_\Z$ is a regular value of $\mu$, then, 
\[
\left( \ind_{S^1} D\right)^{\gamma^*}=\ind D_{\gamma^*}. 
\]
\end{crl}
\begin{proof}
From the second property in Theorem~\ref{main} $\ind_{S^1}(M,V;L)$ is equal to the equivariant index $\ind_{S^1} D$ for the Spin${}^c$ Dirac operator $D$ on $M$ with coefficients in $L$. Then, this is a consequence of Corollaries~\ref{multiplicity formula for equivariant local index} and~\ref{multiplicity formula for equivariant local index 2}.
\end{proof}

If $\gamma^*$ is a critical value of $\mu$, the reduced space $(M_{\gamma^*},\omega_{\gamma^*})$ has more complicated singularities than orbifold singularities in general. Even in this case Meinrenken and Sjamaar defined the Riemann-Roch index for $(M_{\gamma^*},\omega_{\gamma^*})$, and showed that the quantization conjecture still holds for a closed $M$~\cite{MeinSj}. 

We conclude this note with the following question:
\begin{question}
For a critical value $\gamma^*$, is $\ind_{S^1}^{\gamma^*}\left(V_{\gamma^*},V_{\gamma^*}\cap V;L|_{V_{\gamma^*}}\right)$ equal to Meinrenken and Sjamaar's Riemann-Roch index for $(M_{\gamma^*},\omega_{\gamma^*})$ ? 
\end{question}
%%%%%%%%%%%%%%%%%%%%%%%%%%%%%%%%%%%%%%%%%%%%%%%%%%

%\nocite{Green-Wu}%本文で引用されなかったが，文献目録に掲載したい文献
\bibliographystyle{amsplain}
\bibliography{references}
\end{document}